\documentclass{amsart}
\usepackage{amsfonts,amssymb,amscd,amsmath,enumerate,verbatim,calc}
\usepackage[all]{xy}

\newcommand{\CM}{Cohen-Macaulay}

\newcommand{\wrt}{with respect to}

\newcommand{\I}{\mathbb{I} }
\newcommand{\n}{\mathfrak{n} }
\newcommand{\m}{\mathfrak{m} }

\newcommand{\R}{\mathcal{R} }
\newcommand{\C}{\mathcal{C} }
\newcommand{\Z}{\mathbb{Z} }

\newcommand{\rt}{\rightarrow}

\newcommand{\wh}{\widehat }

\newcommand{\Om}{\Omega}

\newcommand{\projdim}{\operatorname{projdim}}

\newcommand{\depth}{\operatorname{depth}}
\newcommand{\Tor}{\operatorname{Tor}}
\newcommand{\Ass}{\operatorname{Ass}}
\newcommand{\ann}{\operatorname{ann}}

\newcommand{\Supp}{\operatorname{\underline{Supp}}}
\newcommand{\Spec}{\operatorname{Spec}}
\newcommand{\CMS}{\operatorname{\underline{CM}}}

\newcommand{\Hom}{\operatorname{Hom}}
\newcommand{\sHom}{\operatorname{\underline{Hom}}}
\newcommand{\Ext}{\operatorname{Ext}}

\theoremstyle{plain}

\newtheorem{theorem}{Theorem}[section]

\newtheorem{lemma}[theorem]{Lemma}

\theoremstyle{definition}

\newtheorem{s}[theorem]{}
\newtheorem{remark}[theorem]{Remark}

\theoremstyle{remark}

\begin{document}

\title[Hilbert polynomials]{Derived functors and Hilbert polynomials over hypersurface rings }
\author{Tony~J.~Puthenpurakal}
\date{\today}
\address{Department of Mathematics, IIT Bombay, Powai, Mumbai 400 076}

\email{tputhen@math.iitb.ac.in}

 \begin{abstract}
Let $(A,\m)$ be a hypersurface local ring of dimension $d \geq 1$ and let $I$ be an $\m$-primary ideal. We show that there is a non-negative integer $r_I$ (depending only on $I$) such that if $M$ is any non-free maximal \CM \ (=  MCM)
$A$-module the  function $n \rt \ell(\Tor^A_1(M, A/I^{n+1}))$ (which is of polynomial type) has degree $r_I$. Analogous results hold for Hilbert polynomials associated to Ext-functors.
Surprisingly a  key ingredient is the classification of thick subcategories of the  stable category of MCM $A$-modules (obtained by Takahashi, see \cite[6.6]{T}).
\end{abstract}
 \maketitle
\section{introduction}
Let $(A,\m)$ be a \CM \ local ring  of dimension $d \geq 1$ and let $I$ be an $\m$-primary ideal. If $N$ is an $A$-module of finite length then $\ell(N)$ denotes its length. Let $M$ be a maximal \CM \ (= MCM) $A$-module. The function $t_I(M, n) = \ell(\Tor^A_1(M, A/I^{n+1}))$ is of polynomial type, see \cite[Corollary 4]{Theo} (also see \cite[Proposition 17]{P1}). Let $t_I^M(z) \in  \mathbb{Q}[z]$ be such that $t_I^M(n) = t_I(M, n)$ for all $n \gg 0$. It is easily shown that $\deg t_I^M(z) \leq d - 1$. In \cite[Theorem 18]{P1} we proved that
$\deg t_\m^M(z)  = d-1$ for any non-free MCM $A$-module. It was also shown that if $I$ is a parameter ideal then $t_I(M, n) = 0$ for all $n \geq 0$, see \cite[Remark 20]{P1}.
In general it is a difficult question to determine the degree of $t_I^M(z)$ and the answer is known only for a few classes of ideals and modules, see \cite[3.5]{KT} for  some examples.
The fact that $\deg t_\m^M(z)  = d-1$  for non-free MCM's has an important consequence in the study of associated graded modules (\wrt \ $\m$) of MCM $A$-modules, see \cite{P3}.

In this paper we prove  few surprising  results.  Recall $A$ is said to be a hypersurface ring if its completion $\wh{A} = Q/(f)$ where $(Q, \n)$ is a regular local ring and $f \in \n^2$ is non-zero. We set degree of the zero polynomial to be $-1$. We show
\begin{theorem}\label{main}
Let $(A,\m)$ be a hypersurface local ring of dimension $d \geq 1$ and let $I$ be an $\m$-primary ideal. Then there is an  integer $r_I \geq -1$ (depending only on $I$) such that if $M$ is any non-free maximal MCM
$A$-module then $\deg t_I^M(z) = r_I$.
\end{theorem}

\s For the Ext functors we prove an analogous result. It is known that if $M$ is a finitely generated $A$-module the function $n \rt \ell(\Ext_A^1(M, A/I^{n+1}))$ is of polynomial type say of degree $s_I^M$, see \cite[Corollary 4]{Theo}.  We prove
\begin{theorem}\label{main-e1}
Let $(A,\m)$ be a hypersurface local ring of dimension $d \geq 1$ and let $I$ be an $\m$-primary ideal. Then there is an  integer $s_I \geq -1$ (depending only on $I$) such that if $M$ is any non-free maximal MCM
$A$-module then $s_I^M = s_I$.
\end{theorem}
It is also known that if $M$ is a finitely generated $A$-module the function $n \rt \ell(\Ext_A^{d+1}(A/I^{n+1}, M))$ is of polynomial type say of degree $e_I^M$, see \cite[Theorem 5]{Theo}.  Let $\Spec^0(A) = \Spec(A) \setminus \{ \m \}$. We prove
\begin{theorem}\label{main-e2}
Let $(A,\m)$ be a hypersurface local ring of dimension $d \geq 1$ and let $I$ be an $\m$-primary ideal. Then there is an  integer $e_I \geq -1$ (depending only on $I$) such that if $M$ is any non-free maximal MCM
$A$-module  free on $\Spec^0(A)$ then $e_I^M = e_I$.
\end{theorem}
See \ref{why} on why in Theorem \ref{main-e2} we need to restrict to the case of MCM modules free on $\Spec^0(A)$ while in Theorems  \ref{main} and \ref{main-e1} we do not have such restriction.

Finally we show
\begin{theorem}\label{te}
(with hypotheses as in \ref{main} and \ref{main-e1}) We have $r_I = s_I$.
\end{theorem}

\s Let $G_I(A) = \bigoplus_{n \geq 0}I^n/I^{n+1}$ be the associated graded ring of $I$. Let $M$ be an $A$-module. Let $G_I(M) = \bigoplus_{n \geq 0}I^nM/I^{n+1}M$ be the associated graded module of $M$ with respect to $I$, considered as a module over $G_I(A)$. We studied the function $t_I(M, n)$ as it is useful in the study of associated graded modules. The following result shows that the reverse direction is very strong at least for hypersurfaces.
\begin{theorem}
\label{ass} Let $(A,\m)$ be a hypersurface local ring of dimension $d \geq 1$ and let $I$ be an $\m$-primary ideal. Suppose there exists a non-free MCM $A$-module $M$ with $G_I(M)$ unmixed and equi-dimensional. Then $r_I = d -1$ or $r_I = -1$. Furthermore if $r_I = -1$ then $\projdim A/I^n < \infty$ for all $n \geq 1$.
\end{theorem}
\emph{Technique used to prove the result:} We first note that the function $t_I(M,n)$ is a function on $\CMS(A)$ the stable category of MCM $A$-modules. We also note that $\CMS(A)$ is a triangulated category \cite[4.4.1]{Buchw}. Let $\CMS_0(A)$ be the thick subctegory of MCM $A$-modules which are free on the punctured spectrum $\Spec^0(A)$ of $A$. The crucial ingredient in our proofs is that $\CMS_0(A)$ has \emph{no} proper thick subcategories, see \cite[6.6]{T}. We first prove Theorem \ref{main} for non-free MCM modules in $\CMS_0(A)$ and then prove for all non-free MCM $A$-modules by using an induction on $\dim \sHom_A(M, M)$. The techniques to prove Theorems \ref{main-e1} and  \ref{main-e2} are similar.

Here is an overview of the contents of this paper. In section two we discuss a few preliminaries that we need.
In section three we prove Theorems \ref{main}, \ref{main-e1}, \ref{main-e2} when $M$ is free on the punctured spectrum of $A$. Finally in section four we prove Theorems \ref{main} and \ref{main-e1}.

\section{Preliminaries}
In this section we discuss a few preliminary results that we need.
We use \cite{N} for notation on triangulated categories. However we will assume that if $\mathcal{C}$ is a triangulated category then $\Hom_\mathcal{C}(X, Y)$ is a set for any objects $X, Y$ of $\mathcal{C}$.

\s \label{t-f} Let $\C$ be a skeletally  small triangulated category  with shift operator $\Sigma$ and let $\I(\C)$ be the set of isomorphism classes of objects in $\C$. By a \emph{weak triangle function} on $\C$ we mean a function $\xi \colon \I(\C) \rt \Z$ such that
\begin{enumerate}
  \item $\xi(X) \geq 0$ for all $X \in \C$.
  \item $\xi(0) = 0$.
  \item $\xi(X \oplus Y) = \xi(X) + \xi(Y)$ for all $X, Y \in \C$.
  \item $\xi(\Sigma X ) = \xi(X)$ for all $X \in \C$.
  \item If $X \rt Y \rt Z \rt \Sigma X $ is a triangle in $\C$ then
   $\xi(Z) \leq \xi(X) + \xi(Y)$.
\end{enumerate}
\s Set $$\ker \xi = \{ X \mid \xi(X) = 0 \}.$$
The following result (essentially an observation) is a crucial ingredient in our proof of Theorem \ref{main}.
\begin{lemma}(see \cite[2,3]{P4} )
\label{ker-lemma}(with hypotheses as above)
$\ker \xi $ is a thick subcategory of $\C$.
\end{lemma}

\s Let $(A,\m)$ be a hypersurface ring and let $I$ be an $\m$-primary ideal in $A$. Let $M$ be a MCM $A$-module.
Set for $n \geq 0$
\begin{align*}
  t_I(M,n) &=  \ell(\Tor^A_1(M, A/I^{n+1})) \\
  s_I(M,n) &= \ell(\Ext^1_A(M, A/I^{n+1}))) \\
  e_I(M, n) &= \ell(\Ext^{d+1}_A( A/I^{n+1}, M))).
\end{align*}
Let $\Omega^i_A(M)$ denote the $i^{th}$-syzygy of $M$.
We prove
\begin{lemma}\label{basic}
(with hypotheses as above)
\begin{enumerate}[\rm (1)]
  \item For all $n \geq 0$  the functions $t_I(-,n), s_I(-,n)$ and $e_I(-,n)$ are functions on $\CMS(A)$
  \item For all $n \geq 0$ we have $t_I(M,n) = t_I(\Omega^1_A(M),n), s_I(M,n) = s_I(\Omega^1_A(M),n)$ and $e_I(M,n) = e_I(\Omega^1_A(M),n)$.
\end{enumerate}
\end{lemma}
\begin{proof}
(1) Let $E = M \oplus F = N \oplus G$ where $F, G$ are free $A$-modules. Then by definition
$t_I(E, n) = t_I(M,n) = t_I(N,n)$. Thus $t_I(-, n)$ is a function on $\CMS(A)$.

 The proof for assertions on $s_I(-, n)$ and $e_I(-, n)$ are similar.

(2) We may assume that $M$ has no free summands. Set $N = \Omega^A_1(M)$.  Let $ 0 \rt N  \rt F \rt M \rt 0$ be the minimal presentation of $M$ with $F = A^r$. Then note as $A$ is a hypersurface ring and $M$ is MCM without free summands we get that a minimal presentation of $N$ is as follows $0 \rt M \rt G \rt N \rt 0$ where $G = A^r$.
By using the first exact sequence we get
\[
0 \rt \Tor^A_1(M, A/I^{n+1}) \rt N/I^{n+1}N \rt F/I^{n+1}F \rt M/I^{n+1} \rt 0.
\]
So we have
\[
t_I(M,n) = \ell(N/I^{n+1}N) + \ell(M/I^{n+1}M) - r\ell(A/I^{n+1}A).
\]
Using the second exact sequence we find that $t_I(M,n) = t_I(N,n)$. The result follows.

The proof for assertions on $s_I(-, n)$ and $e_I(-, n)$ are similar.
\end{proof}
\s\label{ci-lem} The only fact we used in the above Lemma was that $\Omega^2_A(M) \cong M$ if $M$ is a MCM $A$-module with no free summands and that $A$ is Gorenstein. Thus the assertions of the Lemma hold if
$A$ is a local complete intersection and $M$ is a 2-periodic MCM $A$-module.
\section{$\CMS_0(A)$}
In this section we give proofs of Theorem \ref{main}, \ref{main-e1} and \ref{main-e2} when $M$ is free on $\Spec^0(A)$.

\begin{theorem}\label{main-0}
Let $(A,\m)$ be a hypersurface local ring of dimension $d \geq 1$ and let $I$ be an $\m$-primary ideal. Then there is an integer $r_I \geq -1$ (depending only on $I$) such that if $M \in \CMS_0(A)$ is non-zero then $\deg t_I^M(z) = r_I$.
\end{theorem}
\begin{proof}
We first note that for any MCM $M$ we have $\deg t_I^M(z) \leq d - 1$, see \cite[Corollary 4]{Theo}. Recall the degree of the zero polynomial to be $-1$.
Set
$$ r = \max \{ \deg t_I^M(z) \mid  M \in \CMS_0(A) \}. $$
If $r = -1$ then we have nothing to prove. So assume $r \geq 0$. For $M \in \CMS_0(A)$ define
\[
\xi_I(M)  = \lim_{n \rt \infty}\frac{r!}{n^r} t_I(M,n).
\]
We note that $\xi_I(M) \geq 0$ and is zero precisely when $\deg t_I(M, z) < r$. \\
Claim:  $\xi_I(-)$ is a weak triangle function on $\CMS_0(A)$, see \ref{t-f}.\\
Assume the claim for the time being. Then $\ker \xi$ is a thick subcategory of $\CMS_0(A)$.
Also if $\deg t_i^L(z) = r$ then $L \notin \ker \xi$. So $\ker \xi \neq \CMS_0(A)$. As $\CMS_0(A)$ has no proper thick subcategories, see \cite[6.6]{T}, it follows that $\ker \xi = 0$.
Therefore $\deg t_I^M(z) = r$ for all $M \neq 0$ in $\CMS_0(A)$.

It remains to show $\xi_I $ is a weak triangle function on $\CMS_0(A)$. The first three conditions are trivial to satisfy.
By \ref{basic}(2) it follows that $\xi_I(\Om^{-1}_A(M)) = \xi_I(M)$.
Let $ L \rt M \rt N \rt \Om^{-1}(L)$ is a triangle in $\CMS_0(A)$ then note that we have a short exact sequence of $A$-modules
$$ 0 \rt M \rt N \oplus F \rt \Om^{-1}(L) \rt 0, \quad \text{where $F$ is free.}$$
Therefore we have an inequality
$$ t_I(N,n) \leq t_I(M, n) + t_I( \Om^{-1}(L), n).$$
The result follows.
\end{proof}
The following two results can be proved similarly as in \ref{main-0}. We have to use that $\deg s_I^M(z) \leq d -1$ (see \cite[Corollary 4]{Theo}) and that $\deg e_I^M(z) \leq d$ (see \cite[Corollary 7]{Theo}).
\begin{theorem}\label{main-e1-0}
Let $(A,\m)$ be a hypersurface local ring of dimension $d \geq 1$ and let $I$ be an $\m$-primary ideal. Then there is an integer $s_I \geq -1$ (depending only on $I$) such that if $M \in \CMS_0(A)$ is non-zero then $\deg s_I^M(z) = s_I$.
\end{theorem}

\begin{theorem}\label{main-e2-0} (= Theorem \ref{main-e2})
Let $(A,\m)$ be a hypersurface local ring of dimension $d \geq 1$ and let $I$ be an $\m$-primary ideal. Then there is an  integer $e_I \geq -1$ (depending only on $I$) such that if $M \in \CMS_0(A)$ is non-zero then $\deg e_I^M(z) = e_I$.
\end{theorem}

\section{Proofs of Theorem \ref{main} and \ref{main-e1}}
In this section we give proofs of Theorem \ref{main} and \ref{main-e1}. We need a few preliminaries.
\s  Let $M$ be any finitely generated $A$-module.  Set \\ $L_i(M) = \bigoplus_{n \geq 0}\Tor^A_i(M,  A/I^{n+1})$ for $i \geq 0$. Let $\R = A[It]$ be the Rees algebra of $I$. We have an exact sequence of $\R$-modules
$$ 0 \rt \R(1) \rt A[t](1) \rt L_0(A) \rt 0.$$
Tensoring with $M$ yields an inclusion $0 \rt L_1(M) \subseteq \R(1)\otimes M$ and isomorphisms  $L_i(M) \cong \Tor^A_{i-1}(\R(1), M)$ for $i \geq 2$. It follows that $L_i(M)$ are finitely generated $\R$-module for all $i \geq 1$. We note that if $\Om^A_2(M) \cong M$ then we have $L_i(M) \cong L_{i+2}(M)$ for all $i \geq 1$.

\s We also need the following notion. Let $M \in \CMS(A)$. Let
\[
\Supp(M) = \{ P \mid  M_P \ \text{is not free} \ A_P-\text{module} \}.
\]
It is readily verified that $\Supp(M) = V(\sHom(M, M))$.
\begin{proof}[Proof of Theorem \ref{main}]
By Theorem \ref{main-0} we have that there exists $r_I$ such that for any  non-free MCM module $E \in \CMS_0(A)$ we have $\deg t_I^E(z) = r_I$.

Claim: For any non-free MCM $A$-module $M$ we have $\deg t_I^M(z) = r_I$.\\
We prove this assertion by induction on $\dim \Supp(M)$. If $\dim \Supp(M) = 0$ then $M$ is free on $\Spec^0(A)$. In this case we have nothing to show.

Now assume $\dim \Supp(M) > 0$.
As $L_1(M)_n, L_2(M)_n$ have finite length for all $n$ and as $L_1(M), L_2(M)$ are finitely generated $\R$-modules it follows that there exists $l$ such that $\m^l L_i(M)_n = 0$ for all $n$ and for $i = 1, 2$. As $M$ has period two it follows that $\m^l L_i(M)_n = 0$ for all $i \geq 1$ and all $n \geq 0$.

Let
$$x \in \m^l \setminus \bigcup_{ \stackrel{P \supseteq \ann \sHom(M, M)}{P \ \text{minimal}}} P.$$
Let $M \xrightarrow{x} M \rt N \rt \Om^{-1}(A)$ be a triangle in $\CMS(A)$. It is readily verified that support of  $\sHom(N, N)$ is contained in the intersection of support of $\sHom(M,M)$ and $A/(x)$. So $\dim \Supp(N) \leq \dim \Supp(M) -1$. It is also not difficult to prove that $N$ is not free $A$-module. By induction hypotheses  $\deg t_I^N(z) = r_I$. By the structure of triangles in $\CMS(A)$, see \cite[4.4.1]{Buchw}, we have an exact sequence
$0 \rt G \rt N \rt M/xM \rt 0$ with $G$-free. It follows that $L_3(N) = L_3(M/xM)$. We also have an exact sequence
$0 \rt M \xrightarrow{x} M \rt M/xM \rt 0$. As $x \in \ann L_i(M)$ it follows that we have an exact sequence
$$ 0 \rt L_3(M) \rt L_3(M/xM) \rt L_2(M) \rt 0.$$
As the Hilbert function of $L_3(M)$ and $L_2(M)$ are identical, \ref{basic}(2) we get that
$2t_I^M(z) = t_I^N(z)$. It follows that $\deg t_I^M(z) = r_I$. By induction the result follows.
\end{proof}

\s To prove Theorem \ref{main-e1} we need a few preliminaries. Let $M$ be a finitely generated \CM \  $A$-module of dimension $r$. Let \\ $E^i(M) = \bigoplus_{n \geq 0}\Ext^i_A(M, A/I^{n+1})$.
The exact sequence of $\R$-modules
$$ 0 \rt \R(1) \rt A[t](1) \rt L_0(A) \rt 0,$$
induces an isomorphism $E^i(M) \cong \Ext^{i+1}_A(M, \R(1))$ for all $i > d -r$. In particular $E^i(M)$ are finitely generated $\R$-modules for all $i > d -r$.
We note that if $\Om^A_2(M) \cong M$ then we have $E^i(M) \cong E^{i+2}(M)$ for all $i \geq 1$. The proof of Theorem \ref{main-e1} is mostly similar to the proof of Theorem \ref{main}. So we mostly sketch the proof.
\begin{proof}[Sketch of a proof of Theorem \ref{main-e1}]
By Theorem \ref{main-e1-0} we have that there exists $s_I$ such that for any  non-free MCM module $L \in \CMS_0(A)$ we have $\deg s_I^L(z) = s_I$.

Claim: For any non-free MCM $A$-module $M$ we have $\deg s_I^M(z) = s_I$.\\
We prove this assertion by induction on $\dim \Supp(M)$. If $\dim \Supp(M) = 0$ then $M$ is free on $\Spec^0(A)$. In this case we have nothing to show.

Now assume $\dim \Supp(M) > 0$.
As $E^1(M)_n, E^2(M)_n$ have finite length for all $n$ and as $E^1(M), E^2(M)$ are finitely generated $\R$-modules it follows that there exists $l$ such that $\m^l E^i(M)_n = 0$ for all $n$ and for $i = 1, 2$. As $M$ has period two it follows that $\m^l E^i(M)_n = 0$ for all $i \geq 1$ and all $n \geq 0$.
Let
$$x \in \m^l \setminus \bigcup_{ \stackrel{P \supseteq \ann \sHom(M, M)}{P \ \text{minimal}}} P.$$
Let $M \xrightarrow{x} M \rt N \rt \Om^{-1}(A)$ be a triangle in $\CMS(A)$. As before we have $\dim \Supp(N) \leq \dim \Supp(M) -1$ and $N$ is not free. By induction hypotheses  $\deg s_I^N(z) = s_I$. By the structure of triangles in $\CMS(A)$, see \cite[4.4.1]{Buchw}, we have an exact sequence
$0 \rt G \rt N \rt M/xM \rt 0$ with $G$-free. It follows that $E^3(N) = E^3(M/xM)$. We also have an exact sequence
$0 \rt M \xrightarrow{x} M \rt M/xM \rt 0$. As $x \in \ann E^i(M)$ it follows that we have an exact sequence
$$ 0 \rt E^2(M) \rt E^3(M/xM) \rt E^3(M) \rt 0.$$
As the Hilbert function of $E^3(M)$ and $E^2(M)$ are identical, \ref{basic}(2) we get that
$2s_I^M(z) = s_I^N(z)$. It follows that $\deg s_I^M(z) = s_I$. By induction the result follows.
\end{proof}

\begin{remark}\label{why}
  Consider $U^i(M) = \bigoplus_{n \geq 0}\Ext^i_{A}(A/I^{n+1}, M)$. Then for $i \geq d+1$ it is  possible to give a natural $\R$-module structure on $U^i(M)$. However with this structure $U^i(M)$ is NOT finitely generated (note if $xt \in \R_1$ then $x_1t U^i(M)_n \subseteq U^i(M)_{n-1}$). Thus it is not possible to extend the result in \ref{main-e2-0} to all MCM modules.
\end{remark}
\section{Proof of Theorem \ref{te}}
In this section we give a proof of Theorem \ref{te}. To prove the result note that we may assume the residue field $k$ of $A$ is infinite (otherwise take the extension $A[t]_{\m A[t]}$ of $A$).

\s Recall $x \in I$ is superficial with respect to $M$ if there exists $c$ such that $(I^{n+1}M \colon x) \cap I^cM = I^nM$ for all $n \gg 0$. If $\depth M > 0$ then (as $I$ is $\m$-primary) it follows that $x$ is $M$-regular and $(I^{n+1}M \colon x) = I^nM$ for $n \gg 0$. If $k$ is infinite then $I$-superficial elements with respect to $M$ exists. In fact in this case there exists a non-empty open set $U_M$ in the Zariski topology of $I/\m I$ such that if the image of $x$ is in $U_M$ then $x$ is $I$-superficial with respect to $M$.

\s Let $E = \bigoplus_{n \geq 0}E_n$ be a finitely generated graded module over the Rees algebra $\R= A[It]$. Assume $E_n$ has finite length for all $n$. There exists $xt \in \R_1$ such that $xt$ is $E$-filter regular, i.e.,  $(0 \colon_E xt)_n = 0$ for $n \gg 0$. In fact in this case there exists a non-empty open set $U_E$ in the Zariski topology of $I/\m I$ such that if the image of $x$ is in $U_E$ then $xt$ is $E$-filter regular.

\s\label{mod-TE} Assume $\dim A  \geq 2$ and assume $M$ is a MCM $A$-module. Also assume $A$ is Gorenstein.  Let $x \in I$ be such that it is $M \oplus A$-superficial and  $xt$ is $L_1(M)$, $E^i(M)$-filter regular for $i = 1, 2$. We note that $(I^{n+1} \colon x) = I^n$ and $(I^{n+1}M \colon x) = I^nM$ for $n \gg 0$. Set $B = A/(x)$ and $J = I/(x)$.
We have an exact sequence for $n \geq 1$
\[
 0 \rt \ker \alpha_n \rt A/I^n \xrightarrow{\alpha_n} A/I^{n+1} \rt B/J^{n+1} \rt 0,
\]
where $\alpha_n(a + I^n) = xa + I^{n+1}$. We note that $\ker \alpha_n  = (I^{n+1} \colon x)/I^n = 0$ for $n \gg 0$. Set $N = M/xM$ Thus for $n \gg 0$ we have an exact sequence
\begin{align*}
  \Tor^A_1(M, A/I^n) &\xrightarrow{\alpha^1_n} \Tor^A_1(M, A/I^{n+1}) \rt \Tor^B_1(N, B/J^{n+1}) \rt \\
 M/I^nM &\xrightarrow{\alpha^0_n}  M/I^{n+1}M \rt N/J^{n+1}N \rt 0.
\end{align*}
We note that $\ker \alpha^0_n = (I^{n+1}M \colon x)/I^nM = 0$ for $n \gg 0$. Furthermore the map $\alpha^1_n$ is the $n^{th}$-component of the multiplication map by $xt \in \R_1$ on $L_1(M)$. As $xt$ is $L_1(M)$-filter regular it follows that $\ker \alpha^1_n = 0$ for $n \gg 0$. Thus for $n \gg 0$ we have an exact sequence
\s\label{t}
 $$0 \rt  \Tor^A_1(M, A/I^n) \xrightarrow{\alpha^1_n} \Tor^A_1(M, A/I^{n+1}) \rt \Tor^B_1(N, B/J^{n+1}) \rt 0.$$

Similarly as $xt$ is also $E_1(M)$ and $E_2(M)$-filter regular we have an exact sequence
\s\label{e}
$$0 \rt  \Ext^1_A(M, A/I^n) \xrightarrow{\beta^1_n} \Ext_A^1(M, A/I^{n+1}) \rt \Ext_B^1(N, B/J^{n+1}) \rt 0.$$

\s We note that if $A$ is a hypersurface and $I \subseteq \m^2$ is an $\m$-primary ideal then if we go mod an $I$-superficial element $x$ then the resulting ring $B = A/(x)$ is no longer a hypersurface. It is however a complete intersection. Thus we prove the following result:
\begin{theorem}
\label{ci} Let $(A,\m)$ be a complete intersection of dimension $d \geq 1$. Let $I$ be an $\m$-primary ideal. Let $M$ be a non-free MCM $A$-module with bounded betti-numbers. Then $\deg t^M_I(z) = \deg s^M_I(z)$.
\end{theorem}
\begin{proof}
We prove the result by induction on dimension $d$ of $A$. We may assume $M$ has no free summands. As $M$ has bounded betti numbers it follows that $M$ is $2$-periodic, i.e., $\Omega^2_A(M) \cong M$.

 We first consider the case when $d = 1$.  We note that in this case $\deg t^M_I(z) = 0 $ or $ = -1$. Similarly we have  $\deg s^M_I(z) = 0 $ or $ = -1$. We show that  $\deg t^M_I(z) = -1  $ if and only if $\deg s^M_I(z) = -1  $.
Suppose  $\deg t^M_I(z) = -1  $. So $t_I^M(z) = 0$ and therefore
 $t_I(M, n) = 0$ for $n \gg 0$ (say from $n \geq m$). As $t_I(M, n) = t_I(\Omega^1_A(M), n)$ for all $n$ and $M$ is $2$-periodic we have $\Tor^A_i(M, A/I^{n +1}) = 0$ for all $i \geq 1$ and for $n \geq m$. It follows from \cite[Theorem III]{AB} that $\Ext^i(M, A/I^{n +1}) = 0$ for all $i \gg 0$ and for all $n \geq m$. As $s_I(M, n) = s_I(\Omega^1_A(M), n)$ for all $n$ and $M$ is $2$-periodic we have
$s_I(M, n) = 0$ for all $n \geq m$. It follows that $s_I^M(z) = 0$. Thus $\deg s^M_I(z) = -1  $.  The proof of when $\deg s^M_I(z) = -1  $ then $\deg t^M_I(z) = -1  $ is similar and left to the reader.

We now assume that $d \geq 2$ and the result is proved when $\dim A = d -1$.
A proof similar to above shows that $\deg t^M_I(z) = -1  $ if and only if $\deg s^M_I(z) = -1  $. So assume  $\deg t^M_I(z) \geq 0$ (and so  $\deg s^M_I(z) \geq 0 $). We may assume that the residue field of $A$ is infinite. Let $x \in I$ be $M \oplus A$-superficial and also assume that $xt \in \R_1$ is $L_1(M) \oplus E_1(M) \oplus E_2(M)$-filter regular.
We note that $B = A/(x)$ is a complete intersection and $N = M/xM$ is a MCM $B$-module with bounded betti-numbers. Set $J = I/(x)$. So it follows from our induction hypotheses that
$\deg t^N_J(z) = \deg s^M_J(z)$. As $\deg t^M_I(z) \geq 0$, it follows from  \ref{t} that $\deg t^N_J(z) = \deg t^M_I(z) -1 $. Similarly it follows from \ref{e} that  $\deg s^N_J(z) = \deg s^M_I(z) -1 $. The result follows.
\end{proof}

We now give
\begin{proof}[Proof of Theorem \ref{te}]
Let $M$ be non-free MCM $A$-module. Then $t_I = \deg t^M_I(z)$ and $s_I = \deg s^M_I(z).$ Note $M$ is $2$-periodic. By Theorem \ref{ci} the result follows.
\end{proof}

\section{Proof of Theorem \ref{ass}}
In this section we give
\begin{proof}[Proof of Theorem \ref{ass}]
Let $\R = A[It]$.  Recall $L_0(M) = \bigoplus_{n \geq 0}M/I^{n+1}M$. By \cite[5.6]{P3} we have
$$\Ass_\R G_I(M) = \Ass_\R L_0(M).$$
As $A$ is a hypersurface ring, in particular it is a Gorenstein ring. So there exists an MCM $A$-module $N$ and an exact sequence $0 \rt M \rt F \rt N \rt 0$ with $F$ free.
Notice $L_1(N)$ is a $\R$-submodule if $L_0(M)$. In particular if $L_1(M) \neq 0$ then it is necessarily of dimension $d$ by our assumption on $M$. So $t_I = d - 1$ if $L_1(N) \neq 0$.
If $L_1(N) = 0$ then notice for all $n \geq 0$ we have $t_I(N,n) = t_I(\Omega^A_1(N), n) = 0$. So for all $n \geq 0$ we have $\Tor^A_i(N, A/I^{n+1}) = 0$ for all $i \geq 0$. As $N$ does not have finite projective dimension it follows from \cite[1.9]{HW} that $\projdim A/I^{n+1} $ is finite for all $n \geq 0$.
\end{proof}

\begin{remark}
The assumption that $G_I(M)$ is equidimensional is a mild one. It is not difficult to show that if $A$ is also universally catenary then $G_I(M)$ is equi-dimensional for any MCM $A$-module $M$.
The assumption that $G_I(M)$ is un-mixed is not mild. There are plenty of examples when $\depth G_I(M) = 0$.
\end{remark}

\end{document}